\documentclass[11pt]{amsart}
\usepackage{url, 
  amssymb,setspace, mathrsfs,fontenc}
  \usepackage[alphabetic]{amsrefs}
\usepackage{fullpage} 
\usepackage{color}
\usepackage{tikz}
\usepackage[all]{xy}

\usetikzlibrary{decorations.markings,snakes}
\newtheorem{thm}{Theorem}
\newtheorem{lem}[thm]{Lemma}

\newtheorem{cor}[thm]{Corollary}
\newtheorem{defe}[thm]{Definition}
 
\theoremstyle{remark}
\newtheorem{rem}[thm]{Remark}

\DefineSimpleKey{bib}{myurl}
\newcommand\myurl[1]{\url{#1}}
\BibSpec{webpage}{
  +{}{\PrintAuthors} {author}
  +{,}{ \textit} {title}
  +{}{ \parenthesize} {date}
  +{,}{ \myurl} {myurl}
}

\newcommand{\nc}{\newcommand}

\nc{\ssec}{\subsection}

\nc{\on}{\operatorname}

\nc {\cA}{\mathcal{A}}
\nc {\cG} {\mathcal{G}}
\nc {\cK}{\mathcal{K}}
\nc {\cC} {\mathcal{C}}
\nc {\cL} {\mathcal{L}}
\nc {\cE} {\mathcal{E}}
\nc {\cM} {\mathcal{M}}
\nc {\cO}{\mathcal{O}}
\nc {\cF}{\mathcal{F}}
\nc {\cZ}{\mathcal{Z}}

\nc {\bZ}{\mathbb{Z}}

\nc {\uG} {\underline{G}}
\nc {\cB}{\mathcal{B}}
\nc{\rat}{\mathrm{rat}}
      
\nc {\fk}{\mathfrak{k}}
\nc {\fI}{\mathfrak{i}}
\nc {\fg} {\mathfrak{g}}
\nc {\fl} {\mathfrak{l}}
\nc {\fn} {\mathfrak{n}}
\nc {\cP} {\mathcal{P}}
\nc {\fz} {\mathfrak{z}}
\nc {\fc}{\mathfrak{c}}
\nc {\fh}{\mathfrak{h}}
\nc {\fp}{\mathfrak{p}}
\nc {\ft}{\mathfrak{t}}

\nc{\tg} {\mathtt{g}}
\nc {\tP}{\mathcal{P}}
\nc {\hfg} {\widehat{\fg}}
\nc {\hG} {\check{G}}

\nc {\bGm} {\mathbb{G}_m}
\nc{\bC}{\mathbb{C}}
\nc{\bV}{\mathbb{V}}
\nc{\bP}{\mathbb{P}}
\nc{\bA}{\mathbb{A}}
\nc {\bQ}{\mathbb{Q}}
\nc {\bR}{\mathbb{R}}
\nc{\Sl}{\mathfrak{sl}}
\nc{\Gl}{\mathfrak{gl}}
\nc {\So}{\mathfrak{so}}
\nc{\ra}{\rightarrow}

\nc {\tU}{\tilde{U}}
\nc {\tSym}{\widetilde{Sym}}

\nc {\Bun}{\mathrm{Bun}}
\nc {\Fun}{\mathrm{Fun}}
\nc {\crit}{\mathrm{crit}}
\nc {\Ind}{\mathrm{Ind}}
\nc {\Vac}{\mathrm{Vac}}
\nc {\gr}{\mathrm{gr}}
\nc {\ad}{\mathrm{ad}}
\nc {\Sym}{\mathrm{Sym}}
\nc {\Ram}{\mathrm{Ram}}
\nc {\Op}{\mathrm{Op}}
\nc {\Hitch}{\mathrm{Hitch}}
\nc {\fb}{\mathfrak{b}}
\nc{\cDt}{\mathcal{D}^\times}
\nc{\tp}{\mathtt{p}}
\nc {\tk}{\mathtt{k}}
\nc{\cN}{\mathcal{N}}
\nc {\cY}{\mathcal{Y}}
\nc {\loc}{\mathrm{loc}}
\nc {\Res}{\mathrm{Res}}
\nc {\Hom}{\mathrm{Hom}}
\nc {\End}{\mathrm{End}}
\nc {\Perm}{\mathrm{Perm}}
\nc{\Dec}{\mathrm{Dec}}
\nc {\reg}{\mathrm{reg}}
\nc {\GL}{\mathrm{GL}}
\nc{\SL}{\mathrm{SL}}
\nc{\fm}{\mathfrak{m}}
\nc {\fa}{\mathfrak{a}}
\nc {\Spec}{\mathrm{Spec}}
\nc {\cD}{\mathcal{D}}
\nc {\nilp}{\mathrm{nilp}}
\nc {\tN}{\widetilde{\cN}}
\nc {\tO}{\widetilde{\cO}}
\nc {\Sp}{\mathrm{Sp}}
\nc{\cV}{\mathcal{V}}
\nc {\RS}{\mathrm{RS}}
\nc {\ord}{\mathrm{ord}}
\nc {\fu}{\mathfrak{u}}

\begin{document} 
\title{Complete integrability of the parahoric Hitchin system} 

\begin{abstract} 
Let $\cG$  be a parahoric group scheme over a complex projective curve $X$ of genus greater than one. Let $\Bun_\cG$ denote the moduli stack of $\cG$-torsors on $X$. We prove several results concerning the Hitchin map on $T^*\!\Bun_\cG$. We first show that the parahoric analogue of the global nilpotent cone is isotropic and use this to prove that $\Bun_\cG$ is ``very good'' in the sense of Beilinson-Drinfeld. We then prove that the parahoric Hitchin map is a  Poisson map whose generic fibres are abelian varieties. Together, these results imply that the parahoric Hitchin map is a completely integrable system. 
\end{abstract} 

\author{David Baraglia} 
\author{Masoud Kamgarpour}
\author{Rohith Varma}


\address{Department of Mathematics, University of Adelaide} 
\email{david.baraglia@adelaide.edu.au}

\address{School of Mathematics and Physics, The University of Queensland} 
\email{masoud@uq.edu.au}

\address{Institute of Mathematical Sciences, Chennai}
\email{rvarma.kvm@gmail.com}

\maketitle

\tableofcontents

\section{Introduction} \label{s:intro}

The Hitchin system \cite{Hitchin1, Hitchin} lies at the crossroad of geometry, Lie theory, and mathematical physics. It has found remarkable applications to non-abelian Hodge theory \cite{Simpson, SIM} and the Langlands program \cite{BD, Ngo}. The goal of this paper is to prove the Hitchin map continues to have many of its crucial properties if $G$ is replaced by a Bruhat-Tits \emph{parahoric group scheme}. 

In this paper, we work with parahoric group schemes whose generic fibres are simple and simply connected; thus, we do not treat parahoric group schemes whose generic fibres are not split. Moreover, we restrict to characteristic zero. The reason is we rely on \cite{BS} and \cite{BGPM} which assume these restrictions. We expect, however, that most of the results of this paper hold under less rigid assumptions.

\subsection{Parahoric group schemes}  Let $X$ be a smooth projective curve over $\bC$. 
For each $x\in X$, we let ${\cO_x}$ denote the algebra of functions on the completed formal neighbourhood of a point and $\cK_x$ denote the fraction field. If we choose a uniformiser $t$, then ${\cO_x} \simeq \bC[\![t]\!]$ and $\cK_x \simeq \bC(\!(t)\!)$.

Let $G$ be a simple simply connected algebraic group over $\bC$.  An \emph{integral model} for $G$ is a smooth connected affine group scheme over $X$ whose generic fibre is isomorphic to $G$. If $\cG$ is an integral model, then for all but finitely many $x\in X$, called \emph{points of bad reduction} or \emph{ramification points}, one has an isomorphism $\cG({\cO_x})\simeq G({\cO_x})$.

Among all integral models, a special role is played by \emph{parahoric group schemes}. These are integral models $\cG$ such that at the points of bad reduction, $\cG({\cO})$ is a parahoric subgroup of $G(\cK)$. 
To give the reader an idea about parahoric subgroups of the loop group $G(\cK)$, we compare them with their cousins, namely, parabolic subgroups of $G$. For details, cf. \cite{KumarBook}. 
\begin{enumerate} 
\item[(a)]  Recall that parabolic subgroups are exactly those subgroups $P\subset G$ such that $G/P$ is proper. Similarly, parahoric subgroups are exactly those subgroups of the loop group $\tP\subset G(\cK)$ such that the quotient $G(\cK)/\tP$ is \emph{ind-proper}.
\item[(b)] Given a subset of simple roots of $G$, the group generated by the corresponding root subgroups, together with the root subgroups corresponding to all negative simple roots and the maximal torus, is a parabolic subgroup of $G$. This establishes a bijection between subsets of simple roots and conjugacy classes of parabolics. A similar procedure provides a bijection between \emph{proper} subsets of \emph{affine} simple roots and conjugacy classes of parahorics in $G(\cK)$.
\end{enumerate}

Among parahoric subgroups, the most familiar ones are those constructed from parabolics in the following manner. Consider the canonical evaluation map $ev:G({\cO})\ra G$ defined by $t\mapsto 0$. Given a parabolic $P\subset G$, we have that $ev^{-1}(P)$ is a parahoric subgroup of $G(\cK)$. We refer to these as \emph{parahorics of parabolic type}. Equivalently, this is when the subset of affine simple root corresponding to the parahoric is really a subset of the usual (= non-affine) simple roots. 

Given parahoric subgroups $\tP_1,\cdots, \tP_l\subset G(\cK)$ and finitely many points $x_1,\cdots, x_l$ on the curve $X$, there exists integral models $\cG$ over $X$ with bad reduction exactly at the $x_i$'s satisfying $\cG({\cO_{x_i}})\simeq \tP_i$. This is a consequence of the work of Bruhat and Tits \cite{BT1, BT2}; see also \cite{Yu, BS}. Throughout this paper we assume that we only have one point $x \in X$ of bad reduction. This is done in order to keep our notation simple. As we shall see, none of our proofs depend on this assumption, so the result generalises easily to the case of finitely many points.

\subsection{Parahoric Hitchin system} Henceforth, we assume that genus of $X$ is greater than one. 
Let $\cG$ be a parahoric group scheme over $X$ and let $\Omega=\Omega_X$ denote the canonical bundle. Given a $\cG$-bundle $\cE$, one may construct an adjoint bundle $\ad(\cE)$ of $\cE$, which is a vector bundle with a Lie algebra structure on its sheaf of sections. The co-adjoint bundle $\ad^*(\cE)$ is the vector bundle dual to $\ad(\cE)$. A $\cG$-Higgs bundle is a pair $(\cE, \phi)$ consisting of a $\cG$-bundle together with an element $\phi\in \Gamma(X, \ad^*(\cE)\otimes \Omega)$. Let $\cM_\cG$ denote the moduli stack of $\cG$-Higgs bundles. Using Serre duality, one can show that
\[
\cM_\cG\simeq T^*\Bun_\cG,
\]
where $T^*$ should be interpreted as in \cite[\S 1]{BD}.

Now we discuss the Hitchin map. Let $Q_1,\cdots, Q_\ell$ denote algebraically independent homogenous generators of the invariant ring $\bC[\fg]^G$ and let $d_i$ denote the degree of $Q_i$. Using local trivializations of the $\mathcal{G}$-bundle $\mathcal{E}$, one can locally identify $\phi$ with a section of $\fg^* \otimes \Omega(x)$, or via the Killing form, a local section of $\fg \otimes \Omega(x)$. Since changes of local trivialization act on $\phi$ via the adjoint action, we may apply the invariant polynomial $Q_i$ to $\phi$ to obtain a well-defined global section $Q_i(\phi) \in \Gamma( X , \Omega^{d_i}(d_i.x) )$. In this way we may define the Hitchin map on $T^*\Bun_\cG$ as follows: 
\begin{equation} \label{eq:parahoricMap}
(\cE,\phi) \mapsto (Q_1,Q_2, \cdots, Q_\ell)(\phi) \in \bigoplus_i \Gamma(X , \Omega^{d_i}(d_i . x)).
\end{equation} 
This map takes values in $\bigoplus_i \Gamma(X, \Omega^{d_i}(d_i.x))$. However, contrary to the usual Hitchin map, this map is not surjective. Let $\cA_\cG$ denote the closure of its image. This is an affine subvariety of $\bigoplus_i \Gamma(X, \Omega^{d_i}(d_i.x))$.\footnote{Determining the precise description of $\cA_\cG$ is a delicate problem, which we treat elsewhere \cite{BK}.} 

\begin{defe}\label{d:parhit} The parahoric Hitchin map is the map 
\[
h_\cG: T^*\Bun_\cG \ra \cA_\cG.
\]
defined by \eqref{eq:parahoricMap}.  
\end{defe}

 Let $\cN ilp_\cG=h_\cG^{-1}(0)$ denote the parahoric global nilpotent cone. 
Our first main result is that 
 $\cN ilp_\cG$ is isotropic in $T^*\Bun_\cG$. To prove this, we adapt arguments of Ginzburg to the parahoric setting. A key technical tool we use is the description of parahoric torsors via ramified covers \cite{BS}.

 Next, we show that $\Bun_\cG$ is very good in the sense of Beilinson and Drinfeld. In particular, this implies that $\dim(T^*\Bun_\cG)=2\dim(\Bun_\cG)$. 
Our proof is a generalisation of  the argument of Beilinson and Drinfeld (who worked with $\cG=G\times X$) to the parahoric case. The property of being very good means that $\Bun_\cG$ is in a sense not too far from being a Deligne-Mumford stack. The upshot of this is that one can naively apply various notions of symplectic geometry to $T^*\Bun_\cG$ without having to work in the setting of derived geometry. For instance, this allows us to adapt an argument of Bottacin to prove that 
$h_\cG$ is a Poisson map. Here the image of $h_\cG$ is equipped with the trivial Poisson structure, so to say that $h_\cG$ is Poisson is equivalent to saying that the pullbacks of functions on the base are Poisson commuting. 

The fact that the global nilpotent cone is isotropic together with the natural contracting $\bC^*$ action present implies that every fibre of $h_\cG$ has dimension at most $\dim(\Bun_\cG)$. On the other hand, the fact that $h_\cG$ is Poisson implies that every fibre is co-isotropic. Thus, we have a Lagrangian fibration. We prove that this Lagrangian fibration is a completely integrable system, by showing that generic fibres of $h_\cG$ are abelian varieties. A crucial step here is  establishing that the  Hitchin map on 
the moduli of \emph{polystable} parahoric Higgs bundles is proper. We do this by appealing to the parahoric non-abelian Hodge theory established in \cite{BGPM}. It should also be possible to establish properness by proving  semistable reduction for parahoric Higgs bundles. We leave this as an interesting topic for future work.

\subsection{What was known before?} Let $\cG$ be a parahoric group scheme on $X$. If $\cG$ has no point of bad reduction (i.e., $\cG\simeq G\times X$), then we are reduced to the system considered by Hitchin. Next, if the parahoric group scheme is parabolic of Borel type, then the above results are treated in \cite[\S V]{Faltings}. In type $A$, complete integrability of the  parabolic Hitchin map was established by Scheinost and Schottenloher \cite{SS}, but in the setting of semistable moduli spaces as opposed to moduli stacks.  
 
Let us mention that strongly parabolic Higgs bundles have a variant called \emph{weakly} parabolic Higgs bundles. These are pairs $(\cE, \phi)$ consisting of $G$-bundle $\cE$ with parabolic reduction at a point $x$ and an endomorphism $\phi\in \Gamma(  \cE \times_G \fg \otimes \Omega(x))$ whose residue lies in $\fp$. (By comparison, in the strongly parabolic case considered in this article, the residue lies in the nilpotent radical $\fn$.) If $G$ is of type $A$, then Logares and Martens \cite{LogaresMarten} have proved that the Hitchin map on the semistable moduli space of weakly parabolic Higgs bundles defines a generalised completely integrable system, see also \cite{Markman, Bot}. Note that this implies that {\em generic} symplectic leaves are integrable systems. However, this does not imply that any particular symplectic leaf is an integrable system. Thus, one can not deduce integrability of moduli spaces of strongly parabolic Higgs bundles from these results, which in any case have only been proven for parabolics of type $A$.

\subsection{Acknowledgements} The idea of considering the Hitchin map for $\Bun_\cG$ is due to Xinwen Zhu. We thank him for sharing his insights with us. We would also like to thank Dima Arinkin, David Ben-Zvi, Tsao-Hsien Chen, Sergei Gukov, Arun Ram
and Zhiwei Yun for helpful discussions. DB and MK were supported by the Australian Research Council DECRA Fellowships and RV was supported by a post-doctoral fellowship at TIFR, Mumbai.


\section{Parahoric Higgs bundles}

\subsection{Basic Lie theory}\label{ss:lietheory}   
Let $G$ be a simple, simply-connected complex algebraic group of rank $\ell$. Fix a maximal torus $T \subset G$ and let $\ft \subset \fg$ be the corresponding Lie algebras. Let $X(T) = Hom( T , \mathbb{G}_m )$ be the character group, $Y(T) = Hom( \mathbb{G}_m , T )$ the group of $1$-parameter subgroups of $T$ and denote by $( \, , \, ) : Y(T) \times X(T) \to \mathbb{Z}$ the canonical pairing. Let $R \subset X(T)$ be the root system associated to $G$ and fix a choice of positive roots $R^+ \subset R$. The choice of positive roots $R^+$ determines a Borel subgroup $B \subset G$ with unipotent radical $U$. Associated to each root $r \in R$, we have a root space $\fg_r \subset \fg$, a root homomorphism $u_r : \mathbb{G}_a \to G$ and a corresponding subgroup $U_r \subset G$, called the root group corresponding to $r$.

Let $\alpha_1 , \dots , \alpha_{\ell}$ be the simple roots and $\alpha_{max}$ the highest root. We let $\mathfrak{U}$ denote the rational Weyl alcove
\begin{equation*}
\mathfrak{U} := \{ \theta \in Y(T) \otimes \mathbb{Q} \; | \; (\theta , \alpha_{max}) \le 1, \; (\theta , \alpha_i) \ge 0, \; \forall \text{ positive roots } \alpha_i \}
\end{equation*}
and let $\mathfrak{U}^o$ denote the points $\theta \in \mathfrak{U}$ which satisfy $(\theta , \alpha_{max}) < 1$.

\subsubsection{Parahoric subgroups} 
Let $\cO = \bC[\![t]\!]$ be the local ring of formal power series in $t$ and $\cK = \bC(\!(t)\!)$ the fraction field of $\cO$. For any $\theta \in \mathfrak{U}$ and root $r \in R$, consider the integer
\begin{equation*}
m_r(\theta) = -\lfloor ( \theta , r ) \rfloor.
\end{equation*}
The subgroup
\begin{equation}\label{eq:parahoric}
\mathcal{P}_\theta = \langle T(\cO) , \{ z^{m_r(\theta)} U_r(\cO) \}_{r \in R} \rangle \subset G(\cK)
\end{equation}
is a parahoric subgroup in the sense of Bruhat-Tits. Note that $\mathcal{P}_\theta$ and the constructions that follow depend only on the choice of $\theta$ through the facet $\Omega$ of the affine apartment $Y_*(T) \otimes \mathbb{R}$ on which $\theta$ lies, but to keep the notation simple we will avoid any further mention of facets.\footnote{Thus, we are really considering \emph{weighted} parahorics in the sense of \cite{BoalchParahoric}. The weight does not play a role in defining the moduli stack of parahoric torsors, but is essential for defining stability.}

\subsubsection{Parahoric group schemes} 
From Bruhat-Tits theory, we get a smooth affine group scheme $\mathcal{G}_\theta$ over $Spec(\cO)$, which satisfies:
\begin{itemize}
\item[(i)]{$\mathcal{G}_\theta \times_{Spec(\cO)} Spec(\cK) \cong G \times Spec(\cK)$.}
\item[(ii)]{$\mathcal{G}_\theta(\cO) = \mathcal{P}_\theta$.}
\end{itemize}

\subsubsection{Parahorics of parabolic type} 
If $\theta \in \mathfrak{U}^o$ then $\theta$ also determines a parabolic subgroup $P_\theta \subset G$ which is generated by $T$ and the root groups $U_r$ for which $m_r(\theta) = 0$. Moreover, the parahoric $\mathcal{P}_\theta$ and the parabolic $P_\theta$ are related by
\begin{equation*}
\mathcal{P}_\theta = ev^{-1}( P_\theta ),
\end{equation*}
where $ev : G(\cO) \to G$ is the evaluation map sending $t$ to zero. We refer to the case $\theta \in \mathfrak{U}^o$ as the {\em parabolic case} and we call $\mathcal{P}_\theta$ a parahoric subgroup of {\em parabolic type}.

\subsubsection{Parahoric lie algebras}
Associated to $\theta \in \mathfrak{U}$ we also have the corresponding parahoric subalgebra $\tp_\theta \subset \fg(\cK)$, which may be defined as:
\begin{equation}\label{eq:parahoricsubalgebra}
\tp_\theta = \ft(\cO) \oplus \bigoplus_{r \in R} t^{m_r(\theta)} \fg_r(\cO).
\end{equation}

Let $\kappa$ be the Killing form on $\fg$. We define the dual $\tp_\theta^\perp$ as
\[
\tp_\theta^\perp = \{ u \in \fg(\cK) \; | \; \kappa( u , v ) \in \cO \; \; \forall v \in \tp_\theta \}.
\]
We have a natural isomorphism
\begin{equation*}
\Psi : \tp_\theta^\perp \to Hom_\cO( \tp_\theta , \cO)
\end{equation*}
which sends $u \in \tp_\theta^\perp$ to the homomorphism $v \mapsto \kappa(u , v )$. This isomorphism identifies $\tp_\theta^\perp$ with the dual of $\tp_\theta$ as $\cO$-modules. From (\ref{eq:parahoricsubalgebra}) we see that:
\begin{equation*}
\tp_\theta^\perp = \ft(\cO) \oplus \bigoplus_{r \in R} t^{-m_r(\theta)} \fg_{-r}(\cO).
\end{equation*}
Elements of $\tp_\theta^\perp$ should be interpreted as local parahoric Higgs fields.

\subsection{Parahoric torsors}\label{ss:parabolicBundles} 
Fix a marked point $x \in X$ and a parahoric subgroup $\mathcal{P}_\theta \subset G(\cK)$ associated to some $\theta \in \mathfrak{U}$. Recall that $\cO_x$ denotes the completed local ring at $x$. Let $D_x = Spec(\cO_x)$ be the completed formal neighbourhood of $x$. By a {\em parahoric Bruhat-Tits group scheme} on $X$, we mean a group scheme $\mathcal{G}_{X,x,\theta}$ over $X$ such that:
\begin{itemize}
\item[(i)]{$\mathcal{G}_{X,x,\theta} |_{X - x} \cong G \times (X - x)$.}
\item[(ii)]{$\mathcal{G}_{X,x,\theta}|_{D_{x}} \cong \mathcal{G}_{\theta}$.}
\end{itemize}
Such a group scheme exists for any $(X,x,\theta)$ \cite{BS}. To simplify notation we will omit the subscripts $X,x,\theta$ and write $\cG = \cG_{X,x,\theta}$. We will also refer to $\cG$ simply as a {\em parahoric group scheme} on $X$. Denote by $Bun_{\cG}$ the stack of $\mathcal{G}$-torsors on $X$. We note here that the group scheme $\cG$ is not uniquely determined by the data $(X,x,\theta)$, but it follows from \cite{BS} that the corresponding stacks $Bun_{\cG}$ are all equivalent once $(X,x,\theta)$ are fixed. 

In \cite{Heinloth}, it is shown that $Bun_{\cG}$ is a smooth algebraic stack locally of finite type. One can show that the dimension of $Bun_{\cG}$ is given by \cite{BS}:
\begin{equation*}
\dim( Bun_{\cG} ) = (g-1)\dim(G) + \# \{ r \in R^+ \; | \; \langle \theta , r \rangle \neq 0,1 \}.
\end{equation*}

In the parabolic case, i.e. when $\theta \in \mathfrak{U}^o$, torsors for $\cG$ correspond to the more familiar notion of (quasi-)parabolic bundles. We recall the definition \cite{LS}:
\begin{defe} 
Let $P \subset G$ be a parabolic subgroup of $G$. A quasi-parabolic $G$-bundle on $X$ with parabolic structure of type $P$ at $x$ is a pair $(\cE,\cE_x^P)$, where $\cE$ is a principal $G$-bundle on $X$ and $\cE_x^P$ is a $P$-reduction of $\cE$ at $x$; i.e. an element of $\cE_x \times_G (G/P)$. A parabolic $G$-bundle on $X$ is a quasi-parabolic bundle $(\cE , \cE_x^P)$ together with a choice of $\theta \in \mathfrak{U}^o$ such that $P = P_\theta$. One refers to $\theta$ as the weights of the parabolic bundle at $x$.
\end{defe}

\subsection{Parahoric Higgs bundles} \label{ss:HiggsBundles}
Associated to a parahoric group scheme $\cG = \cG_{X,x,\theta}$, we have a bundle of Lie algebras $\on{Lie}( \cG) \to X$ over $X$. This is a vector bundle equipped with a Lie bracket on its sheaf of sections. Away from $x$ it is just given by the trivial bundle $(X - x ) \times \fg$, while in a formal neighbourhood $D_x = Spec( {\cO}_x)$ of $x$, it may be identified with the Lie algebra $\fp_\theta = \on{Lie}(\mathcal{P}_\theta)$, as given by Equation \ref{eq:parahoricsubalgebra}.

Suppose that $\cE \to X$ is a parahoric $\cG$-torsor. We define the adjoint bundle of $\cE$ by
\begin{equation*}
ad(\cE) = \cE \times_{\cG} Lie( \cG),
\end{equation*}
where $\cG$ acts on $Lie(\cG)$ by the adjoint action. It follows that $ad(\cE)$ is a vector bundle over $X$ and is equipped with a Lie bracket on its sheaf of sections. We let $ad^*(\cE)$ denote the dual bundle. 

\begin{defe} Let $\cE \in Bun_{\cG}$ be a parahoric torsor. A parahoric Higgs field is a section $\phi\in \Gamma(X, ad(\cE)^* \otimes \Omega)$. A parahoric Higgs bundle is a pair $(\cE, \phi)$ consisting of a parahoric torsor $\cE$ and a parahoric Higgs field $\phi$. 
\end{defe} 
More generally, to define the stack $\mathcal{M}_{\cG}$ of parahoric Higgs bundles, let $S$ be an arbitrary scheme over $\bC$ and define
\begin{equation*}
\mathcal{M}_{\cG}(S) = \{ (\cE , \phi ) \; | \; \cE \in Bun_{\cG}(S), \; \phi \in \Gamma( X \times S , ad(\cE)^* \otimes \Omega^1_{X \times S / S}) \},
\end{equation*}
with the obvious notion of isomorphism between objects. We recall that $Bun_{\cG}$ is a smooth equidimensional algebraic stack. From standard deformation theory, we deduce an isomorphism of stacks
\begin{equation*}
\mathcal{M}_{\cG} \cong T^*Bun_{\cG}.
\end{equation*}
Thus, $\mathcal{M}_{\cG}$ is a symplectic stack; in particular, we may speak of the Poisson bracket and of isotropic substacks. Note that in the setting of stacks, care has to be taken to study the above notions; see \cite[\S 1]{BD} for details.

In the parabolic case, where $\theta \in \mathfrak{U}^o$ the notion of parahoric Higgs bundles reduces to the more familiar notion of (quasi-)parabolic Higgs bundles, as we now recall. Note that some authors call such Higgs bundles {\em strongly parabolic}.

\begin{defe} Let $(\cE , \cE_x^P)$ be a parabolic bundle. Let $\fp = \on{Lie}(P)$ be the Lie algebra of the parabolic subgroup $P \subset G$ and $\fn \subset \fp$ the nilpotent radical of $\fp$. A parabolic Higgs field is a section $\phi \in \Gamma(X, \cE \times_G \fg^* \otimes \Omega(x) )$ such that $Res_x(\phi) \in \fn$. A quasi-parabolic Higgs bundle is a triple $(\cE, \cE_x^P , \phi)$ consisting of a parabolic bundle $(\cE , \cE_x^P)$ and a parabolic Higgs field $\phi$. A parabolic Higgs bundle is a quasi-parabolic Higgs bundle $(\cE , \cE_x^P , \phi)$ together with a choice of weight $\theta \in \mathfrak{U}^o$ for which $P = P_\theta$.
\end{defe}

\subsection{Relation to equivariant bundles}\label{ss:equivariant}
In this section we remark on the relation between parahoric torsors and parahoric Higgs bundles to equivariant torsors and equivariant Higgs bundles on a Galois cover of $X$.

Suppose that $p : Y \to X$ is a Galois cover with Galois group $\Gamma$ such that $p$ has ramification over the marked point $x \in X$ and nowhere else. Let $y \in Y$ be one of the points lying over $x$ and let $\Gamma_y \subseteq \Gamma$ be the stabiliser, which is cyclic of order $n$ say. Let $g \in \Gamma_y$ be a generator. Following \cite{BS}, we define a $(\Gamma , G)$-bundle on $Y$ to be a principal $G$-bundle $E \to Y$ with $G$ acting on the right, equipped with a lift of $\Gamma$ to an action on $E$ commuting with the $G$-action. In other words, $E$ is a $\Gamma$-equivariant principal $G$-bundle. 

Given a local trivialization of $E$ around $y$, we can identify the fibre $E_y$ with $G$ such that the right action is given by group multiplication on the right. Then $\Gamma_y$ acts on $E_y \cong G$ commuting with the right $G$-action. This implies that there is a homomorphism $\rho_y : \Gamma_y \to G$ such that the action of $\gamma \in \Gamma_y$ on $E_y \cong G$ is left multiplication by $\rho_y(\gamma)$. Since $\Gamma_y  = \langle g \rangle $ is a cyclic group of order $n$, the homomorphism $\rho_y : \Gamma_y \to G$ corresponds to choosing an element $h \in G$ whose order divides $n$, so that $\rho_y(g) = h$. If we choose a different local trivialization of $E$, then $\rho_y$ and hence $h \in G$ change by conjugation. Thus associated to the $(\Gamma , G)$-bundle $E$ is a well defined conjugacy class $\mathcal{C} \subset G$ of finite order and we say that $E$ has {\em type} $\mathcal{C}$. If $\tau \in \mathcal{C}$ is any representative of the conjugacy class $\mathcal{C}$, we will also say that $E$ has type $\tau$. Note that for any given representative $\tau \in \mathcal{C}$, we may choose our local trivialization of $E$ such that $\rho_y(g) = \tau$.

Recall that each element $\theta \in \mathfrak{U}$ of the rational Weyl alcove corresponds to a conjugacy class $\mathcal{C}_\theta \subset G$ of finite order via the exponential map. Let $\tau \in \mathcal{C}_\theta$ be a representative of this conjugacy class. We let $Bun_Y^\tau(\Gamma , G)$ denote the stack of $\Gamma$-equivariant principal $G$-bundles on $Y$ of type $\tau$. According to Balaji-Seshadri \cite[Theorem 5.3.1]{BS}, this stack is equivalent to $Bun_{\mathcal{G}_{X,x,\theta}}$, where $\cG_{X,x,\theta}$ is a parahoric group scheme on $X$. We will often find it useful to work with equivariant bundles on $Y$ in place of parahoric bundles on $X$. We note that  from the equivariant bundle point of view there is clearly no difficulty in allowing $G$ to be any connected reductive group. This will be important for us in the proof of Theorem \ref{t:verygood}. However, when $G$ is non-simply connected the relation between equivariant bundles and parahoric bundles requires more care, see \cite[\S 8.1.13]{BS}.

If $E \to Y$ is an equivariant bundle of type $\tau$, an equivariant Higgs field on $E$ is defined to be a $\Gamma$-equivariant section of $ad^*(E) \otimes \Omega_Y$. We may then speak of equivariant Higgs bundles of type $\tau$. One sees that the stack of all such Higgs bundles is precisely $T^* Bun_Y^\tau(\Gamma , G)$. The equivalence $Bun_{\mathcal{G}_{X,x,\theta}} \cong Bun_Y^\tau(\Gamma , G)$ gives an equivalence $T^* Bun_{\mathcal{G}_{X,x,\theta}} \cong T^* Bun_Y^\tau(\Gamma , G)$. In other words, we obtain an equivalence between parahoric Higgs bundles on $X$ for the group scheme $\mathcal{G}_{X,x,\theta}$ and equivariant Higgs bundles on $Y$ of type $\tau$.  We note that the covering $p : Y \to X$ and the group $\Gamma$ depend on the choice of $\theta \in \mathfrak{U}$, but are not uniquely determined by $\theta$.

\subsection{Parahoric Hitchin map}

Recall that $\cO_x \cong \bC [\![t]\!]$ denotes the completed local ring of $x$ and $D_x = Spec(\cO_x)$. Suppose that $(\cE , \phi)$ is a parahoric Higgs bundle for $\cG = \mathcal{G}_{X,x,\theta}$. Away from $x$, we may identify $\cE$ with a principal $G$-bundle and $\ad(\cE)|_{X - x}$ is the usual adjoint bundle of $\cE |_{X - x}$. Moreover, the Killing form on $\fg$ gives an isomorphism $\ad^*(\cE)|_{X - x} \cong \ad(\cE)|_{X - x}$. Therefore, away from $x$, $\phi$ is a $1$-form valued section of the adjoint bundle. On the other hand, $\phi |_{D_x}$ may be identified with an element of $\tp_\theta^\perp \otimes \Omega^1_{D_x}$. Since we are assuming $\theta \in \mathfrak{U}$, it is easy to see that $m_r(\theta) \le 1$ for all roots $r \in R$ and hence $\tp_\theta^\perp \subset t^{-1} \fg(\cO_x)$. It follows that at all points of $X$, $\phi$ may locally be identified with a section of $\fg \otimes \Omega^1_X(x)$, i.e. $\phi$ is locally a $\fg$-valued $1$-form with at most a first order pole at $x$. A change in local trivialization of $\cE$ acts on $\phi$ via the adjoint action, and hence if $Q_i$ is an invariant polynomial on $\fg$ of degree $d_i$, we may apply $Q_i$ to $\phi$ to obtain a well-defined section $Q_i(\phi) \in \Gamma( X , \Omega^{d_i}(d_i . x))$. As explained in the introduction, this defines a map
\begin{equation*}
(\cE,\phi) \mapsto (Q_1,Q_2, \cdots, Q_\ell)(\phi) \in \bigoplus_i \Gamma(X , \Omega^{d_i}(d_i . x))
\end{equation*} 
where $Q_1 , Q_2 , \cdots , Q_\ell$ are algebraically independent homogeneous generators for the ring of invariant polynomials on $\fg$. This gives a map $\hat{h}_\cG : \mathcal{M}_\cG \to \bigoplus_i \Gamma(X , \Omega^{d_i}(d_i . x))$. We let $\cA_\cG$ denote the closure of the image of the map $\hat{h}_\cG$. Therefore $\cA_\cG$ is an affine subvariety of $\bigoplus_i \Gamma(X, \Omega^{d_i}(d_i.x))$. The map $\hat{h}_\cG$ factors through $\cA_\cG \to \bigoplus_i \Gamma(X , \Omega^{d_i}(d_i . x))$ and thus defines a map
\[
h_\cG: \mathcal{M}_\cG \ra \cA_\cG.
\]
As per Definition \ref{d:parhit}, we call $h_\cG$ the {\em parahoric Hitchin map}.

\section{Global nilpotent cone and the very good property}\label{ss:parabolicHitchin}

We continue to use the notation of the previous section. Thus, $\cG = \cG_{X,x,\theta}$ is a parahoric group scheme over $X$ with one ramification point $x$ such that $\cG(\cO_x)$ is a parahoric subgroup of type $\theta$.

\subsection{Global nilpotent cone}\label{ss:globalnilpotent}

Recall that for any scheme $S$ over $\bC$, we have
\[
\cM_{\cG}(S) = \{(\cE,s) \mid \cE \in Bun_{\cG}(S), \phi \in \Gamma(X \times S, ad(E)^*\otimes \Omega^1_{X\times S/S}) \}.
\]
Since $\cG \mid_{(X-x)} \cong G \times (X-x)$,
we have, via the isomorphism $\mathfrak{g} \cong \mathfrak{g}^*$
induced by the killing form,
$ad(\cE)^* \otimes \Omega^1_{X \times S/S} \mid_{(X -x) \times S} \cong ad(\cE)
\otimes \Omega^1_{X\times S/S} \mid_{(X -x) \times S}$. Thus for
any point $y \in (X-x) \times S$, we have
\[
\phi(y) \in \mathfrak{g} \otimes k(y).
\]

\begin{defe} We say $\phi \mid_{(X-x) \times S}$ is nilpotent if 
\[
\forall y \in (X-x) \times S, \,\ \phi(y) \in \mathfrak{g} \otimes k(y) \,\ \text{is nilpotent}.
\]
The \textit{global nilpotent cone} is the substack
$\cN ilp_{\cG}$ of $\cM_{\cG}$, given by 
\[
\cN ilp_{\cG}(S) = \{ (\cE,\phi) \in \cM_{\cG} \mid \phi \mid_{(X-x) \times S} \text{is nilpotent} \}.
\]
Equivalently, $\cN ilp_{\cG}=h_{\cG}^{-1}(0)$. 
\end{defe} 

We have the following key result:

\begin{thm}
$\cN ilp_{\cG}$ is an isotropic substack of $\mathcal{M}_{\cG}$.
\end{thm}

\begin{proof}
	We use the Balaji-Seshadri description of parahoric torsors as equivariant bundles (\S \ref{ss:equivariant}), and adapt the proof of \cite{Ginzburg} to this setting.  Thus,
	we will restrict ourselves to explaining the necessary modifications involved. Let $\cE$ be a $\mathcal{G} = \cG_{X,x,\theta}$-torsor on $X$
	and let $\phi \in \Gamma(X,ad(\cE)^* \otimes \Omega_X)$ be nilpotent. Let $\mathcal{B} \subset \cG$ be the Borel subgroup scheme,
	defined as the flat closure of $(X-x) \times B$ in
	$\cG$, for a Borel subgroup $B \subset G$.
	From Heinloth~\cite[Lemma 23]{Heinloth}, we have that the natural morphism
	$f: Bun_{\mathcal{B}} \rightarrow Bun_{\cG}$ is surjective.

	Now we can  find a finite Galois cover $p : Y \rightarrow X$,
	with Galois group $\Gamma$, such that the stack of $\Gamma$-equivariant principal
	$G$-bundles on $Y$ of a fixed local type determined by $\theta$, 
	is equivalent to the stack of $\cG$ torsors on
	$X$ \cite[Theorem 5.3.1]{BS}. As explained in \cite{BS}, if we choose one such equivariant bundle $\cF$ on $Y$, we have an isomorphism of group schemes on $X$:
	\[
	\mathfrak{R}^\Gamma_{Y/X}(Ad(\cF)) \cong \cG,
	\]
	where $Ad(\cF)$ denotes the adjoint bundle of groups associated to $\cF$ and $\mathfrak{R}^\Gamma_{Y/X}()$ denotes the invariant direct image functor.
	The equivalence of the categories between equivariant bundles on $Y$ of fixed local type
	and $\mathcal{G}$-torsors is obtained by
	\[
	\mathcal{F}^{'} \mapsto \mathfrak{R}_{Y/X}^{\Gamma}(\mathcal{F}^{'} \wedge^G \cF^{op}),
	\]
	where for any right $G$-torsors $E, F$ on $Y$, $E \wedge^G F^{op}$ denotes the contracted product \cite{BS}:
	\[
	E \wedge^G F^{op} := \frac{E \times_Y F}{(xg , y) \sim (x , yg^{-1})}.
	\]
	Next, we have a $\Gamma$-equivariant $B$ reduction $\cF_B$ of $\cF$, such
	that $\mathfrak{R}_{Y/X}^{\Gamma}(Ad(\cF_B)) \cong \mathcal{B}$. Let $\mathcal{E}^{'}$ be the equivariant bundle on $Y$, which corresponds to $\cE$
	under the equivalence mentioned above. Then we have an equivariant section 
	$\tilde{\phi}$ of $ad(\mathcal{E}^{'}) \otimes \Omega_Y$, which descends to $\phi$. As in \cite{Ginzburg}, there exists an equivariant $B$ reduction of $\mathcal{E}^{'}$ over the generic point
	of $Y$, for which $\tilde{\phi}$ is a section of $\mathfrak{n}_{\mathcal{E}^{'}} \otimes \Omega_Y$, where $\mathfrak{n}$ denotes the nilradical of the Borel subalgebra and
	$\mathfrak{n}_{\mathcal{E}^{'}}$ denotes the corresponding bundle of Lie algebras determined by the $B$ reduction of $\mathcal{E}^{'}$. Since $G/B$ is projective, we can extend this $B$ 
	reduction $\mathcal{E}^{'}_B$ of $\mathcal{E}^{'}$
    to the whole of $Y$. Thus we have a $\mathcal{B}$ reduction $\cE_{\mathcal{B}}$ of
    $\cE$ given by 
    \[
    \cE_{\mathcal{B}} = \mathfrak{R}_{Y/X}^{\Gamma}(\mathcal{E}^{'}_B \wedge^G \cF_B^{op})
    \] 
    such that
    over the generic point, $\phi \in \mathfrak{n}_{\cE_{\mathcal{B}}} \otimes \Omega_X$.
    As in \cite{Ginzburg}, this implies that $f^*(\phi) = 0 \in \Gamma(X,ad(\cE_{\mathcal{B}})^* \otimes \Omega_X)$
    as it vanishes on the generic point.   
    Following the notations of \cite[Lemma 3]{Ginzburg},  for $N_1 = Bun_{\mathcal{B}}$ and $N_2 = Bun_{\cG}$, we have
    $\cN ilp_{\cG} = pr_2(Y_f)$. The rest of the proof
    can be done exactly as in the proof of \cite[Lemma 5]{Ginzburg}.
\end{proof}

\subsection{Parahoric torsors form a very good stack}

Recall \cite[\S 1]{BD}, that a smooth equidimensional algebraic stack $\cY$ is said to be {\em good} if $\dim( T^* \cY) = 2 \dim( \cY)$ and is said to be {\em very good} if
\[
\text{codim}\{ y \in \mathcal{Y} \; | \; \dim( G_y ) = d \} > d \quad \text{for all } d > 0,
\]
where $G_y$ denotes the automorphism group of $y$. Equivalently $\mathcal{Y}$ is very good if $T^*\cY^0$ is dense in $T^*\cY$, where $\cY^0$ is the largest Deligne-Mumford substack of $\cY$. Note that very good implies good. Using the fact that $\cN ilp_{\cG}$ is isotropic, we prove:
\begin{thm}\label{t:verygood}
The stack $Bun_{\cG}$ is very good.
\end{thm}

\begin{proof}
As mentioned in \S \ref{ss:equivariant}, it is more convenient to work with equivariant principal bundles of type $\tau$ on a Galois cover $p : Y \to X$, as opposed to parahoric bundles on $X$. We will allow $G$ to be a connected semisimple group throughout this proof. We will prove that $Bun_Y^\tau(\Gamma , G)$ is very good by adapting the argument used by Beilinson-Drinfeld \cite[\S 2.10.5]{BD} in the non-equivariant setting. The main difference is that we need to be more careful in estimating the dimensions of various spaces involved in the proof. We also need the fact that the nilpotent cone in $T^*Bun_Y^\tau(\Gamma , G)$ is isotropic for any semisimple $G$, which was proved in the previous subsection. 

We must show that $Bun_Y^\tau(\Gamma , G)$ satisfies:
\[
\text{codim}\{ E \in Bun_Y^\tau(\Gamma , G) \; | \; \dim(H^0_\Gamma(Y , ad(E) ) = d \} > d \quad \text{for all } d > 0,
\]
Where $H^0_\Gamma( Y , ad(E) )$ denotes the space of $\Gamma$-equivariant sections of $ad(E)$. This is equivalent to showing that
\begin{equation}\label{39}
\dim( A(G) \setminus A^0(G) ) < \dim( Bun_Y^\tau(\Gamma , G) ),
\end{equation}
where $A(G)$ is the stack of pairs $(E , s )$, $E \in Bun_Y^\tau(\Gamma , G)$, $s \in H^0_\Gamma( Y , ad(E) )$ and $A^0(G) \subseteq A(G)$ is the closed substack of pairs with $s = 0$.

Let $\fc = Spec( \mathbb{C}[\mathfrak{g}]^G)$ be the affine space whose coordinate ring is the ring of invariant polynomials on $\mathfrak{g}$. Recall that
\[
\fc = \mathfrak{g}/\!/G = \mathfrak{t}/W,
\]
where $W$ is the Weyl group. The morphism $\mathfrak{g} \to \fc$ induces a map $H^0_\Gamma( Y , ad(E)) \to Mor( Y , \fc) = \fc$, in other words, we apply invariant polynomials to $s \in H^0_\Gamma( Y , ad(E))$. This gives a natural map $f : A(G) \to \fc$.

For $h \in \mathfrak{t}$, let $\overline{h} \in \fc$ be the image of $h$ under $\mathfrak{t} \to \fc$ and set $A_h(G) = f^{-1}( \overline{h})$. Set $G^h = \{ g \in G \; | \; Ad_g h = h\}$ and $\mathfrak{g}^h = Lie(G^h) = \{ a \in \mathfrak{g} \; | \; [a,h] = 0\}$. Note that $h \in \mathfrak{t}$ is semisimple and so $G^h$ is reductive. In fact $G^h$ is the Levi of a parabolic subgroup of $G$. Denote by $\mathfrak{z}^h$ the center of $\mathfrak{g}^h$. Since $h \in \mathfrak{z}^h$ and there are a finite number of subalgebras of $\mathfrak{g}$ of the form $\mathfrak{z}^h$ as $h$ varies over $\mathfrak{t}$, (\ref{39}) follows from the inequality $\dim( A_h(G) \setminus A^0(G) ) < \dim(Bun_Y^\tau(\Gamma , G) ) - \dim( \mathfrak{z}^h )$. So it is enough to show that
\begin{equation}\label{40}
\dim( A_h(G) ) < \dim( Bun_Y^\tau(\Gamma , G) ) - \dim( \mathfrak{z}^h ) \quad \text{for } h \neq 0
\end{equation}
\begin{equation}\label{41}
\dim( A_0(G) \setminus A^0(G) ) < \dim( Bun_Y^\tau(\Gamma , G) ) .
\end{equation}
Consider $(E , s ) \in A_h(G)$. Then $s \in H^0_\Gamma( Y , ad(E) )$ and $s$ maps to $\overline{h} \in \fc$. Let us write $s = s_s + s_n$, where $s_s$ is semisimple, $s_n$ is nilpotent and $[s_s , s_n] = 0$. With respect to a local trivialisation of $E$ around $y$, we may identify $s_s(y)$ with an element of $\mathfrak{g}$. We may choose the trivialization so that the generator of the stabilizer group $\Gamma_y$ acts as left multiplication by $\tau$. By $\Gamma$-equivariance of $s$, we have that $Ad_{\tau}( s_s(y) ) = s_s(y)$. Choosing a different trivialisation will change $\tau$ and $s_s(y)$ by conjugation. Next we note that since $s_s(y)$ maps to $\overline{h} \in C \cong \mathfrak{t}/W$, we can choose a local trivialisation in which $s_s(y) = h$. Replacing $\tau$ by a conjugate of $\tau$ if necessary, we obtain a trivialization in which $s_s(y) = h$ and the generator of $\Gamma_y$ acts as left multiplication by $\tau$. In this trivialisation the equation $Ad_{\tau}(s_s(y)) = s_s(y)$ becomes $Ad_{\tau}(h) = h$, so $\tau \in G^h$. Denote by $Z^h$ the center of $G^h$. Note that different trivializations in which $s_s(y) = h$ differ by elements of $G^h$. Thus $\tau$ is determined up to conjugation in $G^h$. Therefore, the conjugacy class of $\tau$ in $G^h$ and also in $G^h/Z^h$ is well-defined independent of our choice of local trivialization. Thus it makes sense to speak of the stacks $Bun_Y^\tau( \Gamma , G^h), Bun_Y^\tau(\Gamma , G^h/Z^h)$ and also of $A_h( G^h ), A_0(G^h), A_0(G^h/Z^h)$. We will now argue that (\ref{40}) follows from (\ref{41}) with $G$ replaced by $G^h/Z^h$.

Next we observe that since $h \in \mathfrak{z}^h$, there is an obvious isomorphism $A_0(G^h) \cong A_h(G^h)$ and a further isomorphism $A_h( G^h) \cong A_h(G)$. The latter follows from the fact that if $(E , s) \in A_h(G)$, then $s_s$, the semisimple part of $s$ defines a reduction of structure group to $G^h$.

Observe that there is a natural morphism $\varphi : A_0(G^h) \to A_0(G^h/Z^h)$. A non-empty fibre of $\varphi$ is isomorphic to $Bun_Y^e(\Gamma , Z^h) \cong Bun_X(Z^h)$, where $e$ denotes the identity element of $Z^h$. So
\[
\dim( A_h(G) ) \le \dim( Bun_X(Z^h) ) + \dim( A_0( G^h/Z^h)).
\]
Let us define $e(\theta)$ and $e(\theta , h)$ by:
\begin{equation*}
\begin{aligned}
e(\theta) &= \# \{ r \in R^+ \; | \; \langle \theta , r \rangle \neq 0,1 \}, \\
e(\theta, h) &= \# \{ r \in R^+ \; | \; \langle \theta , r \rangle \neq 0,1 \text{ and } \langle h , r \rangle = 0 \}
\end{aligned}
\end{equation*}
Then one can show \cite{BS} that:
\begin{equation*}
\begin{aligned}
\dim( Bun_Y^\tau(\Gamma , G )) &= (g-1)\dim(G) + e(\theta), \\
\dim( Bun_Y^\tau(\Gamma , G^h/Z^h)) &= (g-1)\dim(G^h/Z^h) + e(\theta , h).
\end{aligned}
\end{equation*}
Now since $\dim( Bun_X(Z^h)) = (g-1)\dim(\mathfrak{z}^h)$ and (\ref{41}) implies that $\dim( A_0(G^h/Z^h)) < (g-1)\dim( \mathfrak{g}^h / \mathfrak{z}^h) + e(\theta , h)$, we get:
\begin{equation*}
\begin{aligned}
\dim( A_h(G) ) &\le (g-1)\dim(\mathfrak{z}^h) + (g-1)\dim( \mathfrak{g}^h / \mathfrak{z}^h) + e(\theta , h) \\
&= (g-1)\dim(G) - (g-1)\dim( \mathfrak{g}/\mathfrak{g}^h) + e(\theta , h) \\
&\le (g-1)\dim(G) + e(\theta) - (g-1)\dim( \mathfrak{g}/\mathfrak{g}^h) \\
&= \dim( Bun_Y^\tau(\Gamma , G)) - (g-1)\dim( \mathfrak{g}/\mathfrak{g}^h).
\end{aligned}
\end{equation*}
But $\dim(\mathfrak{g}/\mathfrak{g}^h) \ge 2 \dim(\mathfrak{z}^h) > \dim(\mathfrak{z}^h)$, if $h \neq 0$. This shows (\ref{40}).\\

To prove (\ref{41}) we will show that if $V \subseteq A_0(G)$ is a locally closed reduced irreducible substack, then $\dim(V) \le \dim( Bun_Y^\tau(\Gamma , G))$ and $\dim(V) = \dim( Bun_Y^\tau(\Gamma , G))$ if and only if $V \subseteq A^0(G)$. Given such a $V$, let $A_0(G)_V$ denote the irreducible component of $A_0(G)$ containing $V$. Let $\cN ilp$ denote the nilpotent cone in $Bun_Y^\tau(\Gamma , G)$. We have shown that $\cN ilp$ is isotropic, hence $\dim(\cN ilp) \le \dim( Bun_Y^\tau(\Gamma , G))$.\\

For $\xi \in H^0_\Gamma( Y , \Omega_Y) \cong H^0(X , \Omega_X)$ consider the morphism $m_\xi : A_0(G)_V \to \cN ilp$ defined by $(E , s) \mapsto (E , s \xi)$. The morphisms $m_\xi$ define $m : A_0(G)_V \times H^0( X , \Omega) \to \cN ilp$. The image of $m$ is contained in some locally closed reduced irreducible substack $Z \subseteq \cN ilp$. If $\xi \neq 0$ then $m_\xi$ induces an embedding $V \to Z_\xi$, where $Z_\xi$ is the closed substack of $Z$ consisting of pairs $(E, \phi)$ such that the restriction of $\phi$ to the subspace $D_\xi = \{ y \in Y \; | \; \xi(y) = 0\}$ is zero. So $\dim(V) \le \dim(Z_\xi) \le \dim(Z) \le \dim(\cN ilp) \le \dim( Bun_Y^\tau(\Gamma , G))$. If $\dim(V) = \dim( Bun_Y^\tau(\Gamma , G))$, then $Z_\xi = Z$ for all non-zero $\xi \in H^0(X , \Omega_X)$. This means that $\phi = 0$ for all $(E , \phi) \in Z$ and therefore $s = 0$ for all $(E , s) \in V$, i.e. $V \subseteq A^0(G)$.
\end{proof}

\section{Poisson commutativity and Lagrangian fibration} 
We continue using the notation of the previous section. Thus, $\cG$ is a parahoric group scheme and $h_\cG: \mathcal{M}_{\cG} \ra \cA_{\cG, X, x}$ is the parahoric Hitchin map.

\subsection{Poisson commutativity}
Note that $h^*_\cG$ is a map 
\[
h^*_\cG: \bC[\cA_{\cG}]\ra \Gamma(\cM_{\cG},\cO).
\]

\begin{thm}\label{t:poissoncommute}
The image of the pullback $h^*_\cG$ consists of Poisson commuting functions. 
\end{thm}

 Our proof is analytic as we will make use of Dolbeault cohomology. One can alternatively give a purely algebraic proof by replacing instances of Dolbeault cohomology with \v{C}ech cohomology. This would lead to a proof along the same lines as given by Bottacin for parabolics of type $A$ \cite{Bot}. However, we found the analytic approach to be more straightforward.
 
\begin{proof} 
Let $(\cE,\phi)$ be a parahoric Higgs bundle. Thus $\phi \in H^0( X , ad(\cE)^* \otimes \Omega )$. If $\alpha$ is a section of $ad(\cE)$ and $\beta$ is a section of $ad(\cE)^*$, then $[\alpha , \beta]$ is a section of $ad(\cE)^*$. We therefore have the following two-term complex:
\begin{equation*}\xymatrix{
ad(\cE) \ar[r]^-{[\phi , \, . \, ]} & ad(\cE)^* \otimes \Omega.
}
\end{equation*}
Borrowing notation from \cite{Bot}, we denote this complex by $[ \phi , \, . \, ]$.

For any parahoric Higgs bundle $(\cE , \phi)$, we have that $\mathbb{H}^0( X , [\phi , \, . \, ])$ is the Lie algebra of infinitesimal automorphisms of $(\cE,\phi)$. Observe that the dual complex $Hom( [\phi , \, . \, ] , \bC)$ tensored by $\Omega$ is canonically isomorphic to $[\phi , \, . \, ]$. Thus, by the extension of Serre duality to hypercohomology, we see that $\mathbb{H}^j(X , [ \phi , \, . \, ] ) \cong \mathbb{H}^{2-j}(X , [ \phi , \, . \, ] )^*$. Therefore $\mathbb{H}^0(X , [\phi , \, . \, ] ) =0$ implies that $\mathbb{H}^2( X , [\phi , \, . \, ] ) = 0$ as well. Thus if $\mathbb{H}^0( X , [\phi , \, . \, ]) = 0$, then around $(\cE , \phi)$, the moduli stack is smooth, Deligne-Mumford and $\mathbb{H}^1( X , [\phi , \, . \, ] )$ is the tangent space at $(\cE,\phi)$. Moreover, Serre duality gives a symplectic pairing on $\mathbb{H}^1( X , [ \phi , \, . \, ] )$, which coincides with the canonical symplectic form $\omega$ on the moduli stack of Higgs bundles.

In Dolbeault cohomology an element of $\mathbb{H}^1(X , [\phi , \, . \, ])$ is represented by a pair $( \alpha , \beta )$, where $\alpha \in \Omega^{0,1}(X , ad(\cE) )$ and $\beta \in \Omega^{0,0}(X , ad(\cE)^* \otimes \Omega)$ satisfying:
\begin{equation*}
\overline{\partial}_{\cE} \beta = [\phi , \alpha],
\end{equation*}
where $\overline{\partial}_{\cE}$ denotes the $\overline{\partial}$-operator on $ad(\cE)^*$. Moreover, the symplectic form $\omega$ on the moduli space of Higgs bundles is given by:
\begin{equation}\label{eq:symplecticform}
\omega( (\alpha_1 , \beta_1) , (\alpha_2 , \beta_2) ) = \int_X \kappa(\alpha_1 , \beta_2 ) - \kappa( \alpha_2 , \beta_1),
\end{equation}
where $\kappa$ denotes the Killing form. This is clear since the above formula is simply the Dolbeault realisation of Serre duality.

Now let $Bun^0_{\cG}$ be the substack of parahoric torsors $\cE$ with $H^0( X , ad(\cE) ) = 0$, i.e. $Bun^0_{\cG}$ is the largest Deligne-Mumford substack of $Bun_{\cG}$. As $Bun_{\cG}$ is very good, we have that  $T^*Bun^0_{\cG}$ is dense in $T^*Bun_{\cG}$. Thus, to show Poisson commutativity of functions in the image of $h^*_\cG$, it is enough to show such functions Poisson commute on $T^*Bun^0_{\cG}$. We therefore restrict attention to parahoric Higgs bundles $(\cE , \phi)$ with $H^0( X , ad(\cE) ) = 0$. This also implies that $\mathbb{H}^0( X , [\phi , \, . \, ]) = 0$.

Let $\rho$ be an invariant polynomial on $\mathfrak{g}$ of degree $d_\rho$. Then for any parahoric Higgs bundle $(\cE,\phi)$ we can regard $\rho(\phi)$ as a holomorphic section of $\Omega^{d_\rho}(d_\rho x)$. Let $\mu \in H^1( X , \Omega^{1-d_\rho}(-d_\rho x))$. By Serre duality we can pair $\rho(\phi)$ with $\mu$ to get a complex number 
\begin{equation*}
f_{\rho , \mu}(\cE,\phi) = \langle \rho(\phi) , \mu \rangle = \int_X \rho(\phi) \mu \in \mathbb{C}.
\end{equation*}
This defines a regular function $f_{\rho , \mu}$ on the Hitchin base and such functions generate the ring of all regular functions on the base. Thus we just need to show that any two such functions $f_{\rho , \mu} , f_{\rho' , \mu'}$ Poisson commute.

To prove Poisson commutativity, we work out the Hamiltonian vector field $X_{\rho , \mu}$ associated to $f_{\rho , \mu}$. By definition this is the vector field whose value at $(\cE,\phi)$ satisfies
\begin{equation}\label{eq:ham}
df_{\rho , \mu}(\cE,\phi)(Y) = \omega( X_{\rho , \mu}(\cE,\phi) , Y)
\end{equation}
for all tangent vectors $Y$ at $(\cE,\phi)$.
Suppose we represent $Y$ in Dolbeault cohomology as $( \alpha_Y , \beta_Y)$ and $X_{\rho , \mu}(\cE,\phi)$ as $(\alpha , \beta)$. Differentiating $\int_X f_{\rho , \mu}(\phi) \mu$ in the $Y$-direction gives
\begin{equation*}
df_{\rho , \mu}(\cE,\phi)(Y) = \int_X d_\rho \rho( \phi , \dots , \phi , \beta_Y) \mu.
\end{equation*}
From (\ref{eq:symplecticform}), we also have:
\begin{equation*}
\omega( X_{\rho , \mu}(\cE,\phi) , Y) = \int_X \kappa( \alpha , \beta_Y) - \kappa( \alpha_Y , \beta).
\end{equation*}
Comparing these expressions, we see that Equation (\ref{eq:ham}) will be satisfied if we choose $\beta = 0$ and let $\alpha$ be defined by:
\begin{equation*}
\kappa( \alpha , \, \_ \, ) = d_\rho \rho( \phi , \phi , \dots , \phi , \, \_ \, ) \mu.
\end{equation*}
In order for this to be a representative in Dolbeault cohomology, we need to check that $(\alpha , \beta)$ chosen in this way is a cocycle, i.e. $[\phi , \alpha] = 0$. For all sections $y$ of $ad(\cE)$, we have:
\begin{equation*}
\begin{aligned}
\kappa( [\phi , \alpha] , y ) &= -\kappa( \alpha , [\phi , y]) \\
&= -d_\rho \rho( \phi , \phi , \dots , \phi , [\phi , y]) \mu \\
&= 0,
\end{aligned}
\end{equation*}
where in the last line we use $ad$-invariance of $\rho$. Hence $[\phi , \alpha] = 0$, as required.

We have shown that the Hamiltonian vector field $X_{\rho , \mu}$ of $f_{\rho , \mu}$ evaluated at the point $(\cE,\phi)$ may be represented by a Dolbeault cocycle of the form $( \alpha_{\rho , \mu} , 0)$. Similarly if we have another such function $f_{\rho' , \mu'}$ then its Hamiltonian vector field $X_{\rho' , \mu'}$ evaluated at $(\cE,\phi)$ may be represented in the form $( \alpha_{\rho' , \mu'} , 0)$. From the definition of the Poisson bracket we have:
\begin{equation*}
\begin{aligned}
\{ f_{\rho , \mu} , f_{\rho' , \mu'} \}(\cE,\phi) &= \omega( X_{\rho , \mu}(\cE,\phi) , X_{\rho',\mu'}(\cE,\phi) ) \\
&= \int_X \kappa( \alpha_{\rho , \mu} , 0) - \kappa( \alpha_{\rho',\mu'} , 0) \\
&= 0.
\end{aligned}
\end{equation*}

This proves Poisson commutativity of the functions $\{ f_{\rho , \mu} \}_{\rho , \mu}$ and hence any two functions in the image of the pullback $h^*_\cG$ will Poisson commute.
\end{proof}

\begin{rem} The only fact about parahoric group schemes used in the above proof is that $\Bun_\cG$ is very good. In other words, Poisson commutativity property of the Hitchin map holds for any integral model $\cG$ such that $\Bun_\cG$ is very good. 
\end{rem} 

\subsection{Lagrangian fibration} In this subsection, we prove that the parahoric Hitchin map 
\[
h_\cG:T^*\Bun_{\cG} \ra \cA_{\cG}
\]
 defines a Lagrangian fibration. More precisely, we have: 

\begin{thm}
\begin{itemize}
\item[(i)]{Any irreducible component of any fibre of $h_\cG$ has dimension $\dim( Bun_{\cG})$. In particular, $\cN ilp_{\cG}$ is a Lagrangian substack of $\cM_{\cG}$.}
\item[(ii)]{The image $\cA_{\cG}$ is irreducible and has dimension $\dim( Bun_{\cG})$.} 
\end{itemize}
\end{thm}

\begin{proof}
Our proof is largely modelled on \cite[Proposition 1]{Ginzburg}. First we show that $\cA_{\cG}$ is irreducible. Indeed, since we know that $Bun_{\cG}$ is very good and connected we deduce that $\cM_{\cG} = T^*Bun_{\cG}$ is irreducible \cite[\S 1]{BD}. It follows that the image of $h_\cG$ is an irreducible topological space and hence its closure, which is $\cA_{\cG}$, is also irreducible.

Since $Bun_{\cG}$ is good, we have $\dim(\cM_{\cG}) = \dim( T^*Bun_{\cG} ) = 2 \dim( Bun_{\cG})$. It follows that each irreducible component of any fibre of $h_\cG$ has dimension greater than or equal to $2 \dim( Bun_{\cG} ) - \dim( \cA_{\cG} )$. In particular, if $\cN$ is an irreducible component of $\cN ilp_{\cG}$, then:
\begin{equation*}
\dim( \cN ) \ge 2 \dim( Bun_{\cG} ) - \dim( \cA_{\cG} ).
\end{equation*}
On the other hand, since $\cN ilp_{\cG}$ is isotropic, we have
\begin{equation*}
\dim( \cN ) \le \dim( Bun_{\cG} ).
\end{equation*}
Note that $\cN ilp_{\cG}$ is Lagrangian if and only if the above is an equality for each irreducible component $\cN$ of $\cN ilp_{\cG}$. 

Next, let $F$ be a generic non-singular fibre of $h_\cG$. Because the Hitchin map is a Poisson map, we have that $F$ is a coisotropic substack of $T^*Bun_{\cG}$ and it follows that every irreducible component of $F$ has dimension at least $\dim( Bun_{\cG})$. Using the natural $\bC^*$-action on $T^*Bun_{\cG}$, we can put the fibre $F$ into a family parametrised by $\mathbb{A}^1$, so that the central fibre is $\cN ilp_{\cG}$ and all other fibres are isomorphic to $F$. One deduces from this that $\dim(F) \le \dim( \cN ilp_{\cG} )$. Combining these dimension estimates, we have
\begin{equation*}
\dim( Bun_{\cG} ) \le \dim(F) \le \dim( \cN ilp_{\cG} ) \le \dim( Bun_{\cG}).
\end{equation*}
We must have equality throughout and hence $\dim(F) = \dim( \cN ilp_{\cG}) = \dim( Bun_{\cG})$. Since the generic fibres of $h_\cG$ have dimension $\dim( Bun_{\cG})$ and $\cM_{\cG}$ is irreducible, we deduce that $\cA_{\cG}$ has dimension $\dim( T^*Bun_{\cG} ) - \dim( Bun_{\cG}) = \dim( Bun_{\cG})$, which proves (ii). Repeating the argument using the $\bC^*$-action shows that every fibre of $h_\cG$ has dimension at most $\dim( \cN ilp_{\cG} ) = \dim( Bun_{\cG})$. On the other hand, any irreducible component of any fibre of $h_\cG$ has dimension at least $\dim( T^* Bun_{\cG}) - \dim( \cA_{\cG}) = \dim( Bun_{\cG})$. Hence any irreducible component of any fibre has dimension equal to $\dim( Bun_{\cG} )$, proving (i).
\end{proof}

\begin{cor}\label{c:flat} 
\begin{enumerate} 
\item[(i)] If $\cA_{\cG}$ is smooth, then the Hitchin map is flat and surjective as a map to $\cA_{\cG}$.
\item[(ii)] More generally, let $\cA^\reg_\cG$ denote the smooth subscheme of $\cA_{\cG}$. Then $h_\cG: h^{-1}_\cG(\cA^{\reg}_\cG) \ra \cA^{\reg}_\cG$ is flat. 
\end{enumerate} 
\end{cor} 

\begin{proof} The very good property implies that $T^*\Bun_{\cG}$ is a local complete intersection, therefore Cohen-Macauley. By the previous theorem, fibres of the the hitchin map $h_\cG$ have the same dimension. As the base is regular, we can use miracle flatness to conclude that $h_\cG$ is flat.

As $h_\cG$ is flat we deduce that the image of $h_\cG$ is an open subset of $\cA_{\cG}$. Let us denote the vector space $\bigoplus_i \Gamma(X, \Omega^{d_i}(d_i.x))$ by $\mathbb{W}$ and let $0 \in \mathbb{W}$ denote the origin. Note that $\cA_{\cG} \subset \mathbb{W}$ is a closed, $\bC^*$-invariant subvariety and $0 \in \cA_{\cG}$. The image of $h$ is open in $\cA_{\cG}$, hence has the form $U \cap \cA_{\cG}$, where $U$ is an open subset of $\mathbb{W}$ containing $0$. It follows that $\bC^* U = \mathbb{W}$. Clearly we also have $\bC^* \cA_{\cG} = \cA_{\cG}$ and hence it follows that $\bC^*( U \cap \cA_{\cG} ) = (\bC^* U) \cap \cA_{\cG} = \mathbb{W} \cap \cA_{\cG} = \cA_{\cG}$. However the image of $h$ is clearly $\bC^*$-invariant, so we get $U \cap \cA_{\cG} = \bC^*(U \cap \cA_{\cG}) = \cA_{\cG}$, proving surjectivity. Part (ii) similarly follows by miracle flatness.
\end{proof} 

\begin{rem} The image of the parahoric Hitchin map is smooth (in fact, an affine space) for many parahorics. However, for certain parahorics in type $D$, this image can be singular; see \cite{BK} for details. 
\end{rem}

\section{Properness and complete integrability} \label{s:integrability} 

Our goal in this section is to show that generic fibres of the parahoric Hitchin map are abelian varieties. For this, we need to show that the Hitchin map is proper over some open subset in the base $\cA_{\cG}$.  In fact, we will show that the Hitchin map on the moduli space of {\em polystable} parahoric Higgs bundles is proper.

\subsection{Properness of the Hitchin map} 

Let $\mathcal{M}_\cG$ denote the moduli stack of $\cG$-Higgs bundles on $X$. The notions of stability and polystability and the corresponding coarse moduli spaces were defined for parahoric bundles in \cite{BS} and extended to parahoric Higgs bundles in \cite{BGPM}. We denote by $\mathcal{M}_{\cG}^{s}$ (resp. $\mathcal{M}_{\cG}^{ps}$) the open substack of stable (resp. polystable) $\cG$-Higgs bundles. Denote by $M^s_{\cG}$ (resp. $M^{ps}_{\cG}$) the underlying coarse moduli spaces of stable (resp. polystable) $\cG$-Higgs bundles on $X$. The Hitchin map $h_\cG : \mathcal{M}_\cG \to \mathcal{A}_\cG$ restricted to $\mathcal{M}^{ps}_{\cG}$ factors through the map $\mathcal{M}^{ps}_\cG \to M^{ps}_{\cG}$ and thus defines a map $h_\cG : M^{ps}_\cG \to \mathcal{A}_{\cG}$.

In this subsection, we prove that: 

\begin{thm}\label{t:properss}
The parahoric Hitchin map $h_{\cG} : M^{ps}_{\cG} \to \mathcal{A}_{\cG}$ on the moduli space of polystable parahoric $\cG$-Higgs bundles is a proper map.
\end{thm}
\begin{proof}
Let $\mathcal{C}_\theta \subset G$ be the conjugacy class corresponding to $\theta$ and let $m$ be the order of $\mathcal{C}_\theta$. We regard $X$ as an orbifold, where the point $x \in X$ has order $m$ and there are no other orbifold points (of course, this can be extended to the case of multiple marked points).

Let $\pi = \pi_1^{orb}(X)$ be the orbifold fundamental group. A presentation of $\pi$ is given by generators 
\[
\alpha_1 , \dots , \alpha_g, \quad \beta_1 , \dots , \beta_g , \quad \gamma,
\]
and relations
\[
[\alpha_1 , \beta_1] \dots [\alpha_g , \beta_g] \gamma = 1, \quad \quad \gamma^m = 1.
\]
In this presentation, $\gamma$ corresponds to a loop around the orbifold point $x$. Let $Hom( \pi , G)$ denote the space of homomorphisms $\rho : \pi \to G$. This is an affine algebraic variety and the group $G$ acts by conjugation. Let $Rep( \pi , G) = Hom( \pi , G)//G$ be the affine GIT quotient. Let $Rep_{\mathcal{C}}( \pi , G) \subset Rep(\pi , G)$ be the closed subvariety of representations $\rho : \pi \to G$ such that $\rho(\gamma) \in \mathcal{C}$.

In \cite{BGPM}, the non-abelian Hodge correspondence is extended to the case of parahoric Higgs bundles. In particular, as a special case of \cite[Theorem 7.8]{BGPM}, we obtain a homeomorphism
\[
M^{ps}_{\cG} \cong Rep_{\mathcal{C}}( \pi , G ),
\]
where $M^{ps}_{\cG}$ is the moduli space of polystable parahoric $\cG$-Higgs bundles on $X$.

Choose an integer $n$ such that $G$ is a closed subgroup of $GL(n,\mathbb{C})$. The inclusion $j : G \to GL(n,\mathbb{C})$ induces a map $Rep( \pi , G) \to Rep(\pi , GL(n,\mathbb{C}))$. By \cite[Corollary 9.16]{SIM}, this map is finite, hence it is also a proper map. Let $\mathcal{C}' \subset GL(n,\mathbb{C})$ be the conjugacy class in $GL(n,\mathbb{C})$ containing $j(\mathcal{C})$. Then we obtain a commutative diagram
\begin{equation*}
\xymatrix{
Rep_{\mathcal{C}}(\pi , G) \ar[r] \ar[d]^-j & Rep(\pi , G) \ar[d]^-j \\
Rep_{\mathcal{C}'}(\pi , GL(n,\mathbb{C})) \ar[r] & Rep(\pi , GL(n,\mathbb{C})) \\
}
\end{equation*}
In this diagram the horizontal maps are closed immersions and the second vertical map is proper, hence the map $j : Rep_{\mathcal{C}}(\pi , G) \to Rep_{\mathcal{C}'}(\pi , GL(n,\mathbb{C}))$ is also proper. The conjugacy class $\mathcal{C}'$ has finite order and therefore is of the following form: there are rational numbers $0 \le \alpha_1 \le \alpha_2 \le \dots \le \alpha_n < 1$ such that $g \in \mathcal{C}'$ if and only if $g$ is conjugate to ${\rm diag}( {\rm exp}(2\pi i \alpha_1) , {\rm exp}(2\pi i \alpha_2) , \dots , {\rm exp}(2\pi i \alpha_n ) )$. We may view $\alpha = (\alpha_1 , \alpha_2 , \dots , \alpha_n)$ as a set of parabolic weights.

By a second application of the non-abelian Hodge correspondence, we have a homeomorphism
\[
Rep_{\mathcal{C}'}(\pi , GL(n , \mathbb{C})) \cong M^{ps}_0( X , x , \alpha , GL(n,\mathbb{C}))
\]
where $M^{ps}_0( X , x , \alpha , GL(n,\mathbb{C}))$ denotes the moduli space of strongly parabolic $GL(n,\mathbb{C})$-Higgs bundles of parabolic degree $0$ and one marked point $x$ with parabolic weights $\alpha$.

As the non-abelian Hodge correspondence commutes with the group homomorphism $j : G \to GL(n,\mathbb{C})$, it is easy to see that we have a commutative diagram of the form
\begin{equation*}
\xymatrix{
M^{ps}_{\cG} \ar[d]^-{h_{\cG}} \ar@{}[r]|*[@]{\cong} & Rep_{\mathcal{C}}(\pi , G) \ar[r]^-j & Rep_{\mathcal{C}'}(\pi , GL(n,\mathbb{C})) \ar@{}[r]|*[@]{\cong} & M^{ps}_0(X , x , \alpha , GL(n,\mathbb{C})) \ar[d] \\
\bigoplus_i \Gamma( X , \Omega^{d_i}(d_i . x) ) \ar[rrr]^-j & & & \bigoplus_{i=1}^n \Gamma( X , \Omega^{i}(i . x ) )
}
\end{equation*}
where the vertical maps are Hitchin maps. We have seen that the upper horizontal map is proper. The lower vertical map is clearly a closed immersion. Therefore to show that the Hitchin map $h_{\cG} : M^{ps}_{\cG} \to \bigoplus_i \Gamma( X , \Omega^{d_i}(d_i . x))$ is proper, it is enough to show that the Hitchin map $M^{ps}_0(X , x , \alpha , GL(n,\mathbb{C})) \to \bigoplus_{i=1}^n \Gamma(X , \Omega^{i}(i.x) )$ is proper. However it is well known that the Hitchin map for parabolic $GL(n,\mathbb{C})$-Higgs bundles is proper \cite{YOK}, and so the proof is complete.
\end{proof}

\begin{rem}
Alternatively, to see that the Hitchin map $M^{ps}_0(X , x , \alpha , GL(n,\mathbb{C})) \to \bigoplus_{i=1}^n \Gamma(X , \Omega^{i}(i.x) )$ is proper, we may identify $M^{ps}_0(X , x , \alpha , GL(n,\mathbb{C}))$ with a moduli space of semistable orbifold $GL(n,\mathbb{C})$-Higgs bundles on $X$. Then by \cite{VAR}, there exists an elliptic surface $S \to X$ over $X$ such that semistable orbifold Higgs bundles on $X$ can be pulled back to give semistable Higgs bundles on $S$ in the usual sense. One sees that this gives rise to a commutative diagram of Hitchin maps
\begin{equation*}
\xymatrix{
M^{ps}_0(X , x , \alpha , GL(n,\mathbb{C})) \ar[r] \ar[d] & M^{ps}( S , GL(n,\mathbb{C}) ) \ar[d] \\
\bigoplus_{i=1}^n \Gamma( X , \Omega^i(i.x)) \ar[r] & \bigoplus_{i=1}^n \Gamma( S , Sym^i( T^*S ) ) 
}
\end{equation*}
where the horizontal arrows are closed immersions. Therefore properness of $M^{ps}_0(X , x , \alpha , GL(n,\mathbb{C})) \to \bigoplus_{i=1}^n \Gamma(X , \Omega^{i}(i.x) )$ follows from properness of the ordinary Hitchin map $M^{ps}( S , GL(n,\mathbb{C}) ) \to \bigoplus_{i=1}^n \Gamma( S , Sym^i( T^*S ) )$ on the surface $S$, which is well known \cite[Theorem 6.11]{SIM}.
\end{rem}

\subsection{Complete integrability} 

\begin{lem}\label{l:aut}
Let $(E , \phi)$ be a polystable $\Gamma$-equivariant $G$-Higgs bundle on $Y$. Under the non-abelian Hodge correspondence, $(E,\phi)$ corresponds to a reductive representation $\rho : \pi_1^{orb}(X) \to G$ of the orbifold fundamental group. Let $Aut_\Gamma(E,\phi)$ be the group of automorphisms of $(E,\phi)$ commuting with the action of $\Gamma$ and let $Aut(\rho) = \{ g \in G \; | \; g \rho g^{-1} = \rho\}$. Then we have an inclusion $Aut_\Gamma(E,\phi) \subseteq Aut(\rho)$.
\end{lem}
\begin{proof}
Let $V$ be any representation of $G$ and set $V_E = E \times_G V$. Then $\phi$ induces a bundle endomorphism $\phi_E : V_E \to V_E \otimes \Omega^1_Y$. Polystability of $(E,\phi)$ implies polystability of $(V_E , \phi_E)$. Thus there exists a hermitian metric $h( \; , \; )$ on $V_E$ (taken to be conjugate linear in the second variable) such that $(V_E , \phi_E)$ satisfies the Hitchin equations
\[
F_A + [\phi_E , \phi^*_E] = 0,
\]
where $F_A$ is the curvature of the Chern connection $\nabla_A = \partial_A + \overline{\partial}$ on $V_E$ induced by $h$. Now suppose $s$ is a holomorphic section of $V_E$ such that $\phi_E(s) = 0$. We claim that $\partial_A s = \phi^*_Es = 0$. To see this consider the following:
\begin{equation*}
\begin{aligned}
0 & =i\int_Y h( F_A s + [\phi_E , \phi^*_E]s , s ) \\
& = i\int_Y h( (\overline{\partial}_A \partial_A s , s) + i\int_Y h( \phi_E \phi^*_E s , s) \\
&= i\int_Y h( \partial_A s , \partial_A s ) - i\int_Y h( \phi^*_E s , \phi^*_E s) \\
&= || \partial_A s ||^2 + || \phi^*_E s ||^2.
\end{aligned}
\end{equation*}
Thus $\partial_A s = \phi^*_E s = 0$, as claimed. Under the non-abelian Hodge correspondence, $(E,\phi)$ corresponds to a flat $G$-connection $\nabla$. The flat connection $\nabla$ acts as the flat connection on $V_E$ given by $\nabla = \nabla_A + \phi_E + \phi^*_E$. Therefore $\nabla s = 0$ and so $s$ is covariantly constant. To prove the lemma, consider the case where $V = Hom( \fg , \fg )$ and $s : ad(E) \to ad(E)$ is an automorphism $s \in Aut_\Gamma(E , \phi)$. The fact that $s$ is an automorphism of $(E,\phi)$ means that $\phi_E (s) = [\phi , s] = 0$. By the above argument, $s$ is an automorphism of $\nabla$ which moreover commutes with $\Gamma$, i.e. $s \in Aut( \rho)$. 
\end{proof}

\begin{lem}\label{l:openzg}
Let $Z(G) \subset G$ be the center of $G$. The open substack of $\cM_\cG$ consisting of pairs $(\cE , \phi)$ whose automorphism group is $Z(G)$ is non-empty.
\end{lem}
\begin{proof}
Let $p : Y \to X$ be a Galois covering of $X$ with Galois group $\Gamma$. Recall that we can identify $\cG$-Higgs bundles on $X$ with $\Gamma$-equivariant $G$-Higgs bundles on $Y$ of type $\tau$ and this is an equivalence of stacks. Therefore it is enough to show that there exists a $\Gamma$-equivariant $G$-Higgs bundle on $Y$ of type $\tau$ whose automorphism group (as a $\Gamma$-equivariant Higgs bundle) is $Z(G)$. By the non-abelian Hodge correspondence, polystable $\Gamma$-equivariant $G$-Higgs bundles of type $\tau$ correspond to reductive representations of the orbifold fundamental group of $X$ for which the holonomy of a loop around $x$ takes values in the conjugacy class containing $\tau$. It is easy to see there exists such representations $\rho$ whose image is Zariski dense in $G$. Such a representation has automorphism group $Z(G)$. By Lemma \ref{l:aut}, the $\Gamma$-equivariant Higgs bundle $(E,\phi)$ on $Y$ corresponding to $\rho$ will have automorphism group contained in $Z(G)$. However, every element of $Z(G)$ clearly acts as a $\Gamma$-equivariant Higgs bundle automorphism of $(E,\phi)$, hence $Aut_\Gamma(E,\phi) = Z(G)$.
\end{proof}

\begin{thm}\label{t:fibres}
Let $h_\cG :T^*\Bun_{\cG} \ra \cA_{\cG}$ denote the parahoric Hitchin map. Then the generic fibres of $h_\cG$ are $Z(G)$-gerbes over a disjoint union of abelian varieties. More precisely, there exists a non-empty Zariski open set $U\subset \cA_{\cG}$ such that the morphism $h^{-1}_\cG(U)\ra U$ is proper and smooth and each fibre is a $Z(G)$-gerbe over a disjoint union of abelian varieties.
\end{thm}

\begin{proof}
Clearly there exists a non-empty smooth Zariski open set $U \subset \cA_{\cG}$ such that the morphism $h^{-1}_\cG(U) \ra U$ is smooth. By Corollary \ref{c:flat} (ii), the restriction of $h_\cG$ to $h^{-1}_\cG(U)$ is flat and in particular an open map. Restricting $U$ if necessary, we can assume that $h_\cG : h^{-1}_\cG(U) \to U$ is surjective. Since $h_\cG |_{h^{-1}_\cG(U)}$ is open, we can further restrict $U$ so that $U \subseteq h_\cG( \cM^{s}_\cG)$. Note that $\cM^s_{\cG}$ is a non-empty open subset. To see this, note that it is known that there exists stable $\cG$-bundles \cite{BS}. If $\cE$ is a stable $\cG$-bundle, then $(\cE , 0 )$ is a stable $\cG$-Higgs bundle, so $\cM^s_\cG$ is non-empty.

Next, let $\mathcal{W} \subseteq \cM_\cG$ be the open substack of pairs $(\cE , \phi)$ whose automorphism group is $Z(G)$, which is non-empty by Lemma \ref{l:openzg}. Therefore, we may assume that $U$ is chosen with $U \subseteq h_\cG(\mathcal{W})$. 

With our assumptions on $U$ we have that every fibre of $h^{-1}_\cG(U)$ contains a stable Higgs bundle and also contains a Higgs bundle with automorphism group $Z(G)$. We will give an argument using Hamiltonian flows to show that in fact every Higgs bundle in $h^{-1}_\cG(U)$ is stable and has automorphism group $Z(G)$. To carry out this argument we identify parahoric $\cG$-Higgs bundles with $\Gamma$-equivariant $G$-Higgs bundles on $Y$ of type $\tau$.

Let $(E,\phi)$ be a $\Gamma$-equivariant $G$-Higgs bundle. Thus $\phi \in H^0( Y , ad(E) \otimes \Omega_Y )^\Gamma$. Since we are working with $\Gamma$-equivariant Higgs bundles, one finds that the tangent space at smooth points is given as follows. Consider the $2$-term complex
\begin{equation*}\xymatrix{
ad(E) \ar[r]^-{[\phi , \, . \, ]} & ad(E) \otimes \Omega_Y.
}
\end{equation*}
Let us denote this complex by $[ \phi , \, . \, ]$. We have that $\Gamma$ acts on this complex and the tangent space to $(E,\phi)$ may be identified with $\mathbb{H}^1( Y , [ \phi , \, . \, ] )^\Gamma$. Let $\rho$ be an invariant polynomial on $\fg$ of degree $d_\rho$ and let $\mu \in H^1( Y , \Omega^{1-d_\rho}(-d_\rho x))^\Gamma$. Since $\Gamma$ is a finite group, by averaging over $\Gamma$, we may represent $\mu$ by a $\Gamma$-invariant section of $\Omega^{0,1}( Y , \Omega^{1-d_\rho}(-d_\rho x) )$. Define a function $f_{\rho , \mu}$ as in the proof of Theorem \ref{t:poissoncommute}. In other words, we use Serre duality to pair $\rho(\phi)$ with $\mu$ to get a complex number 
\begin{equation*}
f_{\rho , \mu}(\cE,\phi) = \langle \rho(\phi) , \mu \rangle = \int_Y \rho(\phi) \mu \in \mathbb{C}.
\end{equation*}

 Let $X_{\rho , \mu}$ be the Hamiltonian vector field associated to $f_{\rho , \mu}$. By essentially the same calculation as given in the proof of Theorem 
 \ref{t:poissoncommute}, we find that $X_{\rho ,\mu}$ evaluated at the point $(E , \phi)$ may be represented in Dolbeault cohomology by a cocycle of the form $(\alpha_{\rho , \mu}(\phi) , 0 )$, where $\alpha_{\rho , \mu}(\phi)$ is defined by:
\begin{equation*}
\kappa( \alpha_{\rho,\mu}(\phi) , \, \_ \, ) = d_\rho \rho( \phi , \phi , \dots , \phi , \, \_ \, ) \mu.
\end{equation*}
Moreover, since $\mu$ was chosen to be $\Gamma$-invariant, $\alpha_{\rho,\mu}$ is also $\Gamma$-invariant. For $t \in \mathbb{R}$, let $(E_t , \phi_t)$ be the $\Gamma$-equivariant Higgs bundle obtained by flowing $(E , \phi)$ along $X_{\rho , \mu}$ for time $t$ and let $\overline{\partial}_{E_t}$ be the $\overline{\partial}$-operator on $ad(E)$ defining the holomorphic bundle $E_t$. Since $X_{\rho , \mu}(E_t , \phi_t) = (\alpha_{\rho , \mu}(\phi_t) , 0)$, we see that $\partial_t \phi_t = 0$, or $\phi_t = \phi$ is constant. Thus $\alpha_{\rho , \mu}(\phi_t) = \alpha_{\rho , \mu}(\phi)$ is $t$-independent and it follows that
\begin{equation*}
(\overline{\partial}_{E_t} , \phi_t ) = (\overline{\partial}_{E} + t \alpha_{\rho , \mu}(\phi) , \phi).
\end{equation*}
Since $\alpha_{\rho , \mu}(\phi)$ is $\Gamma$-invariant, we see that the $\Gamma$-action on $E$ acts by automorphisms of $(\overline{\partial}_{E_t} , \phi_t)$ for any $t$. In this way, the $(\overline{\partial}_{E_t} , \phi_t)$ become $\Gamma$-equivariant Higgs bundles on $Y$. Now suppose that $\psi : E \to E$ is a $\mathcal{C}^\infty$ principal bundle isomorphism that preserves $\phi$. Then as $\psi$ preserves the Killing form $\kappa$ and the invariant polynomial $\rho$, we find that $\psi$ preserves $\alpha_{\rho , \mu}(\phi)$. Therefore $\psi$ preserves $\overline{\partial}_{E}$ if and only if it preserves $\overline{\partial}_{E_t}$ for any given $t \in \mathbb{R}$. In particular this means that the automorphism group of $(E_t , \phi_t)$ is independent of $t$. Suppose we had chosen $(E , \phi)$ with automorphism group $Z(G)$. Then for each $t$, the automorphism group of $(E_t , \phi_t)$ is also $Z(G)$. In a similar manner we see that for each $t$, $(E_t , \phi_t)$ is stable if and only if $(E , \phi )$ is stable.

Observe that the Hamiltonian flows above are complete, i.e., the flows exist for all time $t$. This implies that each connected component of each fibre of $h^{-1}_\cG(U) \to U$ has universal cover given by $\mathbb{C}^n$, such that the Hamiltonian flows are given by translations. In particular, it follows that any two points in the same connected component of a fibre of $h^{-1}_\cG(U) \to U$ are connected by a series of Hamiltonian flows along vector fields of the form $X_{\rho_j , \mu_j}$. Thus any two Higgs bundles in the same connected component of a fibre have the same automorphism group. This shows that $\mathcal{W} \cap h^{-1}_\cG(U)$ is non-empty and is open and closed in the analytic topology. But $h^{-1}_\cG(U)$ is connected since it is a Zariski open subset of $T^*Bun_{\cG}$, which is irreducible as $Bun_{\cG}$ is very good. Therefore $\mathcal{W} \cap h^{-1}_\cG(U) = h^{-1}_\cG(U)$ and so every Higgs bundle in $h^{-1}_\cG(U)$ has automorphism group $Z(G)$, as claimed. A similar argument shows that $\cM^{s}_\cG \cap h^{-1}_\cG(U) = h^{-1}_\cG(U)$, so every Higgs bundle in $h^{-1}_\cG(U)$ is stable.

The underlying space of $\cM^s_\cG$ can be identified with the coarse moduli space $M^s_\cG$ of stable parahoric Higgs bundles. The above results show that $h^{-1}_\cG(U)$ is a $Z(G)$-gerbe over its underlying coarse moduli space, which is an open subset in $M^s_\cG$. Next we use that the Hitchin map on the coarse moduli space of semistable parahoric Higgs bundles is proper to see that the restriction $h_\cG : h^{-1}_\cG(U) \to U$ is proper. 

Let $F$ be the underlying space of a connected component of a fibre of $h^{-1}_\cG(U) \to U$. Then $F$ is compact, since $h$ is proper. The Hamiltonian vector fields $X_1, \dots , X_m$ associated to the coordinates $h_1 , \dots , h_m$ of the Hitchin map give a global frame of commuting holomorphic vector fields on $F$. Since $F$ is compact in the analytic topology, then the existence of such vector fields implies that $F$ is biholomorphic to a complex torus. Moreover we have that $T^*Bun_{\cG}$, $\cA_{\cG}$ and $h_\cG$ are all algebraic, hence $F$ is itself algebraic. It is well known that a complex torus which is algebraic is an abelian variety, so we conclude that $F$ is an abelian variety.
\end{proof}

\begin{rem}
We have shown that each connected component of a generic fibre of the Hitchin map is a $Z(G)$-gerbe over an abelian variety. It is natural to conjecture that this is a trivial gerbe, so that each connected component is the product of an abelian variety with the classifying stack of $Z(G)$. Indeed, this is the case in the usual setting \cite{DP} and is a requirement for Langlands duality of Hitchin systems for Langlands dual groups. Triviality in type $A$ can be seen using the spectral data description of fibres of the Hitchin system \cite{SS}. Showing triviality of the gerbe along the fibres for more general parahorics presumably requires the development of a theory of ``parahoric cameral data". We leave this as an interesting topic for future work.
\end{rem}


  \end{document}